\documentclass[a4paper,11pt,reqno]{amsart}
\usepackage[OT2,OT1]{fontenc}
\usepackage[latin1]{inputenc}
\usepackage{amsmath}
\usepackage{amsthm}
\usepackage{amscd}
\usepackage{amssymb}
\usepackage{extarrows}
\usepackage{pst-node}
\usepackage{tikz-cd}
\usepackage{tikz,tkz-euclide}
\usetkzobj{all}
\usepackage{mathrsfs}
\usepackage{graphicx}
\usepackage{pb-diagram, pb-xy}
\usepackage[all]{xy}
\usepackage{hyperref}
\usepackage{dsfont}
\usepackage{mathtools}
\usepackage{marvosym}
\usepackage{multicol}
\usepackage[new]{old-arrows}
\usetikzlibrary{decorations.markings}
\usepackage{listings}
\usepackage{color}
\definecolor{mygreen}{rgb}{0,0.6,0}
\definecolor{myyellow}{rgb}{1,1,0.95}
\definecolor{mygray}{rgb}{0.5,0.5,0.5}
\definecolor{mymauve}{rgb}{0.58,0,0.82}
\makeatletter
\def\@tocline#1#2#3#4#5#6#7{\relax
  \ifnum #1>\c@tocdepth 
  \else
    \par \addpenalty\@secpenalty\addvspace{#2}%
    \begingroup \hyphenpenalty\@M
    \@ifempty{#4}{%
      \@tempdima\csname r@tocindent\number#1\endcsname\relax
    }{%
      \@tempdima#4\relax
    }%
    \parindent\z@ \leftskip#3\relax \advance\leftskip\@tempdima\relax
    \rightskip\@pnumwidth plus4em \parfillskip-\@pnumwidth
    #5\leavevmode\hskip-\@tempdima
      \ifcase #1
       \or\or \hskip 1em \or \hskip 2em \else \hskip 3em \fi%
      #6\nobreak\relax
    \hfill\hbox to\@pnumwidth{\@tocpagenum{#7}}\par
    \nobreak
    \endgroup
  \fi}
\makeatother

\newcommand{\sqed}{\hfill{$\square$}}

\newcommand\cyr{%
\renewcommand\rmdefault{wncyr}%
\renewcommand\sfdefault{wncyss}%
\renewcommand\encodingdefault{OT2}%
\normalfont\selectfont}
\DeclareTextFontCommand{\textcyr}{\cyr}

\DeclareMathOperator{\card}{card}
\DeclareMathOperator{\Spec}{Spec}

\DeclareMathOperator{\Tor}{Tor}

\newcommand{\C}{\mathbb{C}}

\newcommand{\N}{\mathbb{N}}
\newcommand{\IP}{\mathbb{P}}
\newcommand{\Q}{\mathbb{Q}}
\newcommand{\R}{\mathbb{R}}

\newcommand{\Z}{\mathbb{Z}}

\newcommand{\SO}{\mathscr O}
\newcommand{\setmi}{\smallsetminus}

\newtheorem{thm}{Theorem}[section]
\newtheorem{cla}[thm]{Claim}

\newtheorem{hyp}[thm]{Hypothesis}
\newtheorem{lem}[thm]{Lemma}
\newtheorem{prop}[thm]{Proposition}

\theoremstyle{definition}
\newtheorem{defn}[thm]{Definition}

\newtheorem{rem}[thm]{Remark}

\hypersetup{
    colorlinks=false,
    pdfpagemode=FullScreen,
}

\newcommand{\charef}[1]{Chapter \ref{#1}}

\newcommand{\defnref}[1]{Definition \ref{#1}}
\newcommand{\equref}[1]{Equation (\ref{#1})}

\newcommand{\lemref}[1]{Lemma \ref{#1}}
\newcommand{\propref}[1]{Proposition \ref{#1}}

\newcommand{\thmref}[1]{Theorem \ref{#1}}

\renewcommand{\to}{\xrightarrow{\quad}}
\newcommand{\apc}{\rightarrow}
\renewcommand{\mapsto}{\longmapsto}

\tikzset{double line with arrow/.style args={#1,#2}{decorate,decoration={markings,%
mark=at position 0 with {\coordinate (ta-base-1) at (0,1pt);
\coordinate (ta-base-2) at (0,-1pt);},
mark=at position 1 with {\draw[#1] (ta-base-1) -- (0,1pt);
\draw[#2] (ta-base-2) -- (0,-1pt);
}}}}

\newcommand\lhexbr{%
  \mathopen{%
    \mspace{1mu}%
    \text{\lhexbra}%
    \mspace{1mu}%
  }%
}

\newcommand\rhexbr{%
  \mathopen{%
    \mspace{1mu}%
    \text{\rhexbra}%
    \mspace{1mu}%
  }%
}

\newcommand\lhexbra{%
    \tikz[line cap=round,x=1ex,y=1ex,line width=0.3pt,baseline={([yshift={-1.6ex}]current bounding box.north)}]
    {\draw (0.4,2) -- (0.0,1.5) -- (0.0,0.5) -- (0.4,0);}%
}

\newcommand\rhexbra{%
    \tikz[line cap=round,x=1ex,y=1ex,line width=0.3pt,baseline={([yshift={-1.6ex}]current bounding box.north)}]
    {\draw (-0.4,2) -- (0.0,1.5) -- (0.0,0.5) -- (-0.4,0);}%
}

\newcommand{\imof}[1]{\lhexbr#1\rhexbr}

\usetikzlibrary{decorations.pathmorphing}

\swapnumbers
\date{December 10, 2019.}

\title{On the Smoothness of Severi Variety}%
\author{Xiao Yang}%

\begin{document}
\numberwithin{equation}{section}

\begin{abstract}
In this paper, we aim at giving a rigorous proof of the statements on the smoothness and the dimension of Severi varieties where there are gaps in the proofs in some standard literature. The method is a mixture of algebraic and analytic methods.
\end{abstract}

\maketitle

\vspace{7mm}
\tableofcontents

\newpage

\section{Introduction}
\label{cha:0}

Throughout, the base field is $\C$. A variety is a quasiprojective, integral and separated scheme of finite type over $\Spec(\C)$. A curve is a projective one-dimensional scheme. Hence, all curves are irreducible and reduced.

In the Appendix F in the lecture script \cite{SEV21}, F. Severi tried to establish the theory of \textit{parametrization space} of degree $d$ plane curves in which each point is associated to a principal homogeneous ideal of degree $d$ in $\C[X,Y,Z]$. Two principal homogeneous ideals are the same if and only if the generators are associated. Explicitly, $$\left(\sum_{\substack{i+j+k=d\\i,j,k\in\N}}c_{ijk}X^iY^jZ^k\right)\mapsto[c_{ijk}]\in\IP^{L_d},$$
where $L_d=\frac{d(d+3)}{2}$. He then specified a closed subset $V_{d,g}$ in the parametrization space correspond to the Zariski closure of the collection of all plane curves of degree $d$ with \textit{geometric genus} $g$. The statement that he tried to prove was $V_{d,g}$ is an algebraic variety of dimension $3d-1+g$ which is nowadays known as the \textit{Severi variety}.

Despite the direct and natural definition of the Severi varieties, the statement is, however, not at all easy to show. This is one of the typical mathematical problem which seems innocuous but actually has led to a lot of dilemmas. He considered the subset $D_{d,n}$ of $V_{d,g}$ which consists of plane curves with only $n$ nodes where $n$ is linked to $g$ by the degree-genus formula. We also refer $D_{d,n}$ as Severi variety. He claimed that $D_{d,n}$ is open, dense in $V_{d,g}$, and smooth of dimension $3d+g-1$. Effectively, he claimed that $D_{d,n}$ was connected. However, his proof did not convince the public. It had remained unproven until the paper \cite{HAR86} was published where J. Harris continued the proof idea of Severi. But this will not be concerned as the goal of this paper although it is interesting and highly nontrivial.

This author of this paper was motivated by the literature \cite{HAM98} P24 where the argument for the smoothness and the dimension of $D_{d,n}$ is not clear or cogent, either. Possibly, a crucial step is missing.  The proof of smoothness and the dimension of $D_{d,n}$ is actually a step stone of the claim that Severi wanted, hence, crucial. Therefore, the aim of this paper is to explain where the gaps in \cite{HAM98} are and to fill them. 

This paper is arranged as allowing: in \charef{cha:1}, we will review some notions and compute the geometric genus of nodal plane curves; in \charef{cha:2}, we will define some sets related to the problem, prove that they are constructible, address the smoothness of these sets, and eventually compute the dimension; in \charef{cha:3}, we will discuss what is missing in the literature \cite{HAM98} and try to fill the gaps.

\section{Genus of Nodal Plane Curves}
\label{cha:1}

\begin{defn}
A \textit{numerical polynomial} is a polynomial $p\in\Q[T]$ such that for all $n\in\Z$, $p(n)\in\Z$.
\end{defn}

In \cite{HAR77} the definition is different but actually the same because:

\begin{lem}
$f\in\Q[T]$ is numerical if and only if there exists $N\in\N$ such that  
$f(n)\in\Z$ for all $n\in\N$ with $N\leqslant n\leqslant N+\deg(f)$.
\end{lem}
\begin{proof}
A numerical polynomial must satisfy the condition.\sqed

For the other direction, prove by induction on the degree of $f$. When degree is 0 it is clear. Set $g(n):=f(n+1)-f(n)$. Then $\deg(g)\leqslant\deg(f)-1$ and $g(n)\in\Z$ for all integers $N\leqslant n\leqslant N+\deg(f)-1$. Thus $g(n)\in\Z$ for all $n\in\Z$. Hence, once $f(n)\in\Z$, we will have $f(n\pm 1)\in\Z$. But $f(N)\in\Z$, so $f(n)\in\Z$ for all $n\in\Z$.
\end{proof}

\begin{defn}
Let $R:=\C[X_0,...,X_n]$ be the polynomial ring over a field $\C$. Given a finitely generated $R$-module $M$, the \textit{Hilbert polynomial} of $M$ is a numerical polynomial $h_M$ such that $h_M(n)=\dim_{\C}(M_n)$ for all $n\in\Z$. The \textit{Hilbert polynomial} of a projective variety $X$ over $\Spec(\C)$ is a numerical polynomial $h_X$ such that $h_X(n)=\dim_{\C}(S(X)_n)$ for all $n\in\Z$ where $S(X)$ is the homogeneous coordinate ring of $X$.
\end{defn}

Apparently for $X:=\IP^m$, due to some basic results in combinatorics, $$h_X(n)=\begin{pmatrix}n+m\\m\end{pmatrix}.$$

\begin{defn}
Let $C$ be a curve, we define its \textit{arithmetic genus} to be
$$p:=\dim_{\C}(H^1(C,\SO_C)).$$
The \textit{geometric genus} of $C$ is defined to be the arithmetic genus of the normalization of $C$.
\end{defn}

\begin{rem}
Note that for any smooth curve $C$, the geometric and arithmetic genera coincide. Another way to compute the (geometric) genus of $C$ which is smooth can be $$g=\dim_{\C}(\Gamma(C,\omega_C))$$ where $\omega_C$ is the canonical sheaf of $C$. More details can be found in \cite{HAR77} Chapter II.8 and Chapter IV.1.
\end{rem}

\begin{prop}
Let $C$ be a curve in $\IP^2_k$ of degree $d$. Then its arithmetic genus is $$p=\frac{(d-1)(d-2)}{2}.$$
And $1-h_C(0)=p$.
\end{prop}
\begin{proof}
Assume $C$ generated by $f$ does not pass through $[1:0:0]$. We partition the curve into two parts $A:=\{X_1\neq0\}\cap C$ and $B:=\{X_2\neq0\}\cap C$. Compute the \v{C}ech complex of
$$0\to\Gamma(A,\SO_C)\oplus\Gamma(B,\SO_C)\xlongrightarrow{\varphi}\Gamma(A\cap B,\SO_C)\to0$$
which will give $$\dim_\C(H^0(C,\SO_C))=1$$
and 
$$\dim_\C(H^1(C,\SO_C))=\frac{(d-1)(d-2)}{2}.$$
On the other hand, there is an exact sequence of graded $R$-modules
$$0\to R(-d)\xlongrightarrow{\cdot f}R\to R/(f)\to0$$
with $R:=k[X_0,X_1,X_2]$. Hence, the Hilbert polynomial of $S(C):=R/(f)$ is
$$h_{S(C)}(n)=\begin{pmatrix}n+2\\2\end{pmatrix}-\begin{pmatrix}n-d+2\\2\end{pmatrix}.$$
which is the same as the Hilbert polynomial of $C$. Hence,
$$1-h_C(0)=\begin{pmatrix}-d+2\\2\end{pmatrix}=\frac{(2-d)(1-d)}{2}=\frac{(d-1)(d-2)}{2}=p.$$
Hence, the result.
\end{proof}

\begin{prop}
Let $C$ be a nodal plane curve of degree $d$. Let $E\to C$ be the normalization of $C$. Resolution of singularity will branch out each node to two preimages for each node. A divisor $\Delta:=\sum p_i$ on $E$ is summing up these. Embed $C\longhookrightarrow\mathbb{P}^2$. Then by successive composition of these two maps, we have a closed immersion $\mu:E\to\mathbb{P}^2$.

There is an \emph{adjunction formula} $$\mu^*\SO_{\mathbb{P}^2}(d-3)\otimes\SO_E(-\Delta)\cong\omega_E$$
or written more compactly $$\SO_{E}(d-3)(-\Delta)\cong\omega_E.$$
Then the geometric genus of $C$ is 
\begin{equation}
\label{eq:1.1}
g=\frac{(d-1)(d-2)}{2}-n.
\end{equation}
\end{prop}
\begin{proof}
Compute the Poincar\'e residue. Details refer to \cite{ACG85} P50.
\end{proof}

Because the normalization of curves always exists and give a non-negative geometric genus. Therefore, the number of nodes of a nodal plane curve has an upper bound.

\section{The Parametrization Spaces}
\label{cha:2}

We continue what we have discussed in the introduction.

\begin{defn}
\label{defn:2.1}
Let $\IP^{L_d}$ be the equation space of polynomial of degree $d$. And let $A_d$ be the set of all irreducible curves; let $A_{d,g}$ be the set of all irreducible reduced curves of geometric genus $g$; let $B_{d,n}$ be the set of all irreducible reduced curves with at least $n$ (distinct) nodes; let $D_{d,n}$ be the set of all irreducible reduced curves with exactly $n$ (distinct) nodes, cf. \cite{HAR77} P32.
\end{defn}

$g,n$ can not be negative and have an upper bound. One can show that the genus or the number of nodes can be any arbitrary as long as they do not exceed the upper bound. So throughout, we assume $n,d,g$ are all non-negative and satisfy \equref{eq:1.1}. Apparently given $d$ we have
$$D_{d,n}\subseteq A_{d,g}\subseteq A_d\subseteq\IP^{L_d},$$
and
$$D_{d,n}=B_{d,n}\smallsetminus B_{d,n+1}$$
for appropriate $g$ and $n$.

Our goal is to show that $D_{d,n}$ is smooth and of dimension $L_d-n$ at each point. By dimension or smoothness at a point in a (constructible) set, we mean that in the closure of the set. 

\begin{cla}
Let $C\subseteq\IP^2$ be a plane curve defined by an equation $f$ and $x\in C$ is a node if and only if $\mathrm{Hess}(f)$ is invertible at $x$.
\end{cla}
\begin{proof}
Without loss of generality, we can prove this at the origin on an affine chart $(X,Y)$. By definition, $f=aX^2+bXY+cY^2+$(high degree terms). If $a$ or $c$ is non-zero, say $a\neq0$, $aX^2+bXY+cY^2=$ has discriminant $b^2-4ac$. The determinant of Hessian matrix of $f$ is $4ac-b^2$. Hence, the invertibility of the Hessian matrix is exactly the same as the condition of $(0,0)$ being a node. And if $a=c=0$ and $b\neq0$, $bXY$ is factorized into two directions, i.e., along $X$ and $Y$.
\end{proof}

\begin{prop}
Retain the notations in \defnref{defn:2.1}, $A_d$ is open in $\IP^{L_d}$.
\end{prop} 
\begin{proof}
If a polynomial is reducible or non-reduced. Then we can factor it into polynomials of strictly smaller degrees. This will gives a morphism of projective varieties $\IP^{L_{d_1}}\times\IP^{L_{d_2}}\to\IP^{L_{d}},([a_{ij}],[b_{lk}])\mapsto[\sum_{i+l=r,j+k=s}a_{ij}b_{lk}]$. Hence, the image is closed due to Chevalley's theorem. $A_d$ corresponds to the image taken away all these closed subsets.
\end{proof}
\begin{defn}

Given $d,n$ we use the space $\IP^{L_d}\times(\IP^2)^n$ to indicate the the coefficients of curves with the first component and the second one denotes the possible singular points.
Consider the following equations 
\begin{equation}
\begin{cases}
\label{eq:1.2}
\varphi_k(c_{ij},X_k,Y_k)=\sum_{ij}c_{ij}X_k^iY_k^j\overset{!}{=}0\\
\psi_k(c_{ij},X_k,Y_k)=\sum_{ij}ic_{ij}X_k^{i-1}Y_k^j\overset{!}{=}0\\
\theta_k(c_{ij},X_k,Y_k)=\sum_{ij}jc_{ij}X_k^iY_k^{j-1}\overset{!}{=}0\\
\end{cases}
\end{equation}
where $1\leqslant k\leqslant n$.
We denote the closed set defined by these equations in $\IP^{L_d}\times(\IP^2)^n$ by $E_{d,n}$. Let $$H_{d,n}:=\{(C,p_1,...,p_n)|C\in B_{d,n}\text{ and with distinct nodes } p_1,...,p_n\}$$
and
$$K_{d,n}:=\{(C,p_1,...,p_n)|C\in D_{d,n}\text{ and with distinct nodes } p_1,...,p_n\}.$$
\end{defn}

\begin{prop}
\label{prop:5.18}
$H_{d,n}$, $K_{d,n}$, $B_{d,n}$ and $D_{d,n}$ are constructible.
\end{prop}
\begin{proof}
On a patch of affine charts of $(\IP^2)^n$,
\equref{eq:1.2} can be translated as the condition that points passing through the curve, and the points are singular. But we do not get all desired curves. First, we should consider all affine charts and the total number of equation is added up to $3n$. Second, we should also carve out the diagonals on $(\IP^2)^n$ since the points are chosen to be distinct.
Precisely, for each $n$ set $$E'_{d,n}:= E_{d,n}\cap(\IP^{L_d}\times((\IP^2)^n\smallsetminus\mathrm{Diagonals}))$$
which is the set of $(C,p_1,...,p_n)$ with $C$ singular at these pairwise distinct points.

Continue writing another other equation
$$\left(\sum_{i,j}i(i-1)c_{i,j}X_k^{i-2}Y_k\right)\left(\sum_{i,j}j(j-1)c_{i,j}X_kY_k^{j-2}\right)-$$
$$\left(\sum_{i,j}ijc_{ij}X_k^{i-1}Y_k^{j-1}\right)^2\overset{!}{=}0$$
to carve out further closed sets in order to get $H_{d,n}$. We have only imposed open relations. Therefore, $H_{d,n}$ is locally closed and the closure of $H_{d,n}$ is $E_{d,n}$. Project $H_{d,n+1}$ to $\IP^{L_d}\times(\IP^2)^n$ by omitting the last coordinate and take the difference between $H_{d,n}$ and the image of $H_{d,n+1}$. Thus, we get the constructible set $K_{d,n}$. Finally, project $H_{d,n}$ and $K_{d,n}$ to $\IP^{L_d}$. Consequently, we get $B_{d,n},D_{d,n}$ which are all constructible.
\end{proof}

Now that they are constructible sets, then the smoothness or dimension are enough to be checked on the closed points which is mentioned in \cite{HAR77} P177. And the smoothness is coincide with the version in differential geometry.

\begin{prop}
\label{prop:5.19}
$K_{d,1}\subseteq\IP^N\times\IP^2$ is smooth of dimension $L_d-1$.
\end{prop}
\begin{proof}
Any $(C,p)$ is a curve with a node at $p$ which satisfies \equref{eq:1.2}. Pick a chart such that the node is at $(0,0)$. The Jacobian matrix at the point with respect to only $(c_{00},X_1,Y_1)$ is 
$$\begin{pmatrix}1 & 0\\ * & \mathrm{Hess}(C)_p \end{pmatrix}.$$
It defines a codimension $3$ condition in $\IP^{L_d}\times\IP^2$ since the Hessian is invertible. Indeed, on a small enough analytically open ball around that curve along the manifold, they all correspond to curves has singular points with non-degenerating Hessian. Then along each point the dimension of tangent space is constantly of dimension $L_d+2-3=L_d-1$. $H_{d,1}$ is locally closed, actually open in $E_{d,n}$. Hence, we can apply constant rank theorem which gives the desired dimension.
\end{proof} 

\begin{prop}
\label{prop:5.20}
Let $p_1,...,p_n$ be nodes on a nodal plane curve of degree $d$. They will impose independent condition on system $|\SO_{\IP^2}(d-3)|$ of plane curves of degree $d-3$.
\end{prop}
\begin{proof}
Let $S$ be the vector space of all plane curves determined by degree $d-3$ homogeneous polynomials which passing through $p_1$, ..., $p_n$. The dimension of $S$ is at least $\frac{(d-2)(d-1)}{2}-n$. Due to a statement in \cite{FUL98} P165, the vector space is a subspace of $\Gamma(C,\SO_{C}(d-3)(-\Delta))$. Per definition $\SO_{C}(d-3)$ is the pullback sheaf of $\SO_{\mathbb{P}^2}(d-3)$ and $\SO(-\Delta)$ gives the sheaf of those sections vanishing at these specific point. On the other hand, $$\Gamma(C,\SO_{C}(d-3)(-\Delta))\cong\Gamma(C,\SO_{C}(d-3)(-\Delta))\cong H^0(C,\omega_C):=g.$$
Equation $(1)$ gives exactly the same bound. Hence, the result.
\end{proof}
\begin{prop}
\label{prop:2.8}
Retain the notations in \defnref{defn:2.1}, $K_{d,n}$ is smooth and the dimension of tangent space at each point is $L_d-n$.
\end{prop}
\begin{proof}
By \propref{prop:5.18}, we know $K_{d,n}$ is constructible. Therefore, it suffices to show that it is of dimension $L_d-n$. Pick any nodal plane curve $(C,p_1,...,p_n)$ with $n$ nodes. We will work on an affine chart of $\IP^2$ such that all points are given by $p_l=(x_l,y_l)$ on the affine plane and take an affine chart of $\IP^{L_d}$ such that the coefficient of the polynomial which determines the curve is give by $c_{ij}$. Set $$\tilde{c}_l:=\sum_{i,j}x^i_ly^j_lc_{ij}.$$
Due to \propref{prop:5.20}, this is surjective. Then we can extend this to a bijective linear map by introducing more variables to the list of $\{\tilde{c}_l\}$. This gives a map $$F:(\tilde{c}_l,X,Y)\mapsto(c_{ij},X,Y)\mapsto(\varphi_k,\psi_k,\theta_k)\in\C^{3n}.$$
One computes the Jacobian,
$$J_{(C,p_1,...,p_n)}(F)=\begin{pmatrix}
1 & 0 & 0 & 0 & \\
* & \mathrm{Hess}(f)_{p_1} & * & 0 & ...\\
0 & 0 & 1 & 0 &\\
* & 0 & * & \mathrm{Hess}(f)_{p_2} &\\
 &  & ... &  & \\
\end{pmatrix}.$$
$\partial_{\tilde{c}_l}\varphi_k=\delta_{lk}$ because we can rewrite $$\sum_{i,j}c_{ij}X_k^iY_k^j=\sum_{i+j>0}\bar{c}_{ij}(X_k-x_k)^i(Y_k-y_k)^j+\tilde{c}_k$$
where $\bar{c}_{ij}$ is uniquely determined by $c_{ij}$, hence, by $\tilde{c}_{l}$.\\
Apply Gauss elimination, we get 
$$\begin{pmatrix}
1 & 0 & 0 & 0 & \\
* & \mathrm{Hess}(f)_{p_1} & 0 & 0 & ...\\
0 & 0 & 1 & 0 &\\
0 & 0 & * & \mathrm{Hess}(f)_{p_2} &\\
 &  & ... &  & \\
\end{pmatrix}.$$
We can see that this is full rank, i.e., $3n$. By constant rank theorem, $K_{d,n}$ is smooth and the tangent space is of dimension $L_d+2n-3n=L_d-n$. 
\end{proof}
 
\section{Resolve the Problem}
\label{cha:3}

But we are not done yet: what we proved was the space $K_{d,n}$ not $D_{d,n}$. An educated intuition leads to that $D_{d,n}$ is smooth and of dimension $L_d-n$ because $K_{d,n}$ is nothing but an $(n!)$-fold covering of $D_{d,n}$ by permuting all nodes. But a very a bad thing will happen: even an injective immersion can fail to be an embedding. This was the problem in \cite{HAM98} P31 where the author simply claimed that the local isomorphism on tangent would implies global embedding. The treatment of the author for local isomorphism was correct. But he still assumed the smoothness of $D_{g,n}$. A priori, it is not given! Till this point, we have not show it the smoothness yet.

It seems that $\pi:E_{d,n}\to\overline{D_{d,n}}$ is \'etale at each point in $K_{d,n}$. $\pi$ is locally of finite presentation. Fiber at each point just has finitely many points. Remain to show is the flatness of the morphism. However, without smoothness at the points $D_{d,n}$, it seems to be impossible to deduce. Indeed, consider the normalization of a curve with just one node and remove one of the preimage of the node. The morphism is not \'etale, but the fiber of morphism always have one point.

In order to cover the missing point, we will adapt to the differential geometry reflection of this problem. Consider the famous Hopf-fibration
$$U(1)\cong S^1\to S^{2n+1}\xlongrightarrow{p}\C\IP^n.$$
Precisely, $p:(z_0,z_1,...,z_n)\mapsto[z_0:z_1:...:z_n]$ with a fiber $S^1$ since $(z_0,z_1,...,z_n)$ and $(e^{i\theta}z_0,e^{i\theta}z_1,...,e^{i\theta}z_n)$ both lie on the curve and will correspond to the same point in $\C\IP^n$. Multiplying by $e^{i\theta}$ is a Lie group action. The Lie group $U(1)$ acts properly, freely and isometrically on $S^{2n+1}$ which gives a quotient manifold structure on $\C\IP^n$.

In the way, thinking $(S^{2n+1},g_{\mathrm{sph}})$ as Riemmanian manifold, we have a unique way to let $\C\IP^n$ carry a metric $h$ making $p$ into a Riemmanian submersion which is the \textit{Fubini-Study} metric, $d_{n}^{FS}$. Rather than computing the distance by the metric, we should think of it more geometrically. $S^{2n+1}\to\C\IP^n$ is a Riemannian submersion. Because The fiber at each point (e.g., $[1:0:...:0]$) has a Hopf fiber diffeomorphic to $S^1$, locally we can lift a geodesic of $\C\IP^n$ upon $S^{2n+1}$. The geodesic is always perpendicular to the Hopf fiber and is always a minimizing geodesic locally. The exponential map $\exp_p:\mathrm{T}_p\C\IP^n\to\C\IP^n$ gives a local diffeomorphism within the ball of radius $\pi/2$. But outside ball it fails to be injective which means there are multiple geodesics hitting the same point on the equator.

\begin{lem}
\label{lem:5.24}
For any point $p\in\C\IP^n$, the cut locus of $p$ is $\{q\in\C\IP^n\mid d^{FS}_n(p,q)=\pi/2\}$ which is isometric to $\C\IP^{n-1}$. Explicitly, if $p=[a_0:...:a_n]$, then
$$\mathrm{Cut}(p)=\left\{[b_0:...:b_n]\in\C\IP^n\left|\,\,\sum a_ib_i=0\right.\right\}.$$
And the distance between $[1:a_1:....:a_n]$ and the hyperplane $\{[0:b_1:....:b_n]\in\C\IP^n\}$ is $$\arccos\left(\sqrt{\frac{|a_1|^2+...+|a_n|^2}{1+|a_1|^2+...+|a_n|^2}}\right).$$
\end{lem}
\begin{proof}
In order to prove the two statements rigorously, one needs to compute the Jacobi field (details can be found in \cite{BES78} P82). But geometrically, we have already described the geodesics of the manifold: we lift the geodesic of $\C\IP^n$ upon $S^{2n+1}$ where we can compute the distance easily.
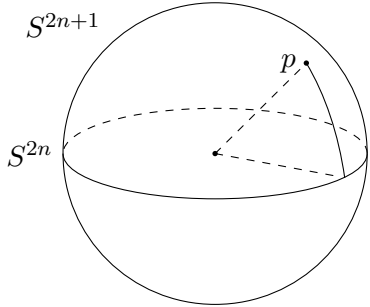
\begin{center}
\begin{tikzpicture}
\shade[ball color = gray!40, opacity = 0] (0,0) circle (2cm);
\fill[fill=white] (-2,0) circle (1pt)node[left]{$S^{2n}$};
\draw (0,0) circle (2cm);
\draw (-2,0) arc (180:360:2 and 0.6);
\draw[dashed] (2,0) arc (0:180:2 and 0.6);
\fill[fill=black] (0,0) circle (1pt);
\draw[dashed] (0,0) -- node[above]{} (1.6,-0.3);
\draw[dashed] (0,0) -- node[above]{} (1.2,1.2);
\fill[fill=black] (1.2,1.2) circle (1pt) node[left]{$p$};
\draw (1.2,1.2) arc(45:13.4:1.9 and 3.2);
\fill[fill=white] (-2,2) circle (1pt)node[below]{$S^{2n+1}$};
\end{tikzpicture}
\end{center}
The arclength can be easily computed can be computed through Euclidean geometry.
\end{proof}
The significance of the lemma above shows the on the parametrization space of curves, if there is a sequence of curves $\{C_i\}$ converging to a curve $C$ (under $d_n^{FS}$) and the coefficient of a term of $C$ is non-zero, then there exists a $N$ and all $i>N$ the coefficient of the that term of $C_i$ must be also non-zero. In the following, ``near" simply means the ``lying in a sufficiently small neighborhood" which can be circumspectly formulated by $\epsilon$-$\delta$ language of, albeit pedantic. And for now on, a curve is just synonym to the polynomial defining it.

\begin{defn}
Let $C_1,C_2$ be two curves of degree $d$ in $\IP^{L_d}$. The distance between two curves is defined to be the distance between the corresponding points in the complex projective space given by the underlying metric above, $d_{L_d}^{FS}$.
\end{defn}

\begin{lem}
For any Cauchy sequence of curves degree $d$ curves $\{C_i\}_{i\in I}$ with $\lim_{i\apc\infty} C_i\apc C$, we have for big $N$, $C_i$ and $C$ can be contained on the same chart and all terms in $C$ appear also in $C_i$ for all $i>N$. Hence, fix a choice of coefficients on $C$ the coefficients of $C_i$, by some rescaling, the coefficient of each term of $C_i$ is near (in the Euclidean sense) to the same term of that of $C$ for all big $i$.
\end{lem}
\begin{proof}
If either the first two statements is not true, by \lemref{lem:5.24} we have the distance between infinitely many $C_i$ and $C$ must be always greater than a positive value.

Now, we can study the curves on an affine chart. Then the last statement is automatically true because the induced topology are equivalent.
\end{proof}

The lemma above just says if a sequence of curves converges to a curve, then their coefficients must be ``close" to each other by an appropriate scaling and all terms of the limit must always appear in the sequence after some point.

\begin{lem}
\label{lem:5.27}
$f\in\C[z]$ be a polynomial of degree $n$. Say $f(z)=a_nz^n+a_{n-1}z^{n-1}+...+a_0$ $(a_n\neq0)$ has the set of roots $R(f):=\{p_1,...,p_n\}$ $($need not be distinct$)$. Then for every $m\geqslant n$ and every $\epsilon>0$, there exists $\delta> 0$ such that every polynomial $g(z)=b_mz^m+b_{m-1}z^{m-1}+...+m_0$ $(b_m\neq0)$ having the set of roots $R(g)=\{q_1,...,q_m\}$ $($need not be distinct, either$)$. If assume for all $k=1,...,n ,|a_k-b_k|\leqslant \delta$ and all $k>n,|b_k|<\delta$, we then can find $S\subseteq R(g)$ such that $\card(S)=n$ and $$\sum_{q\in S}d(q,R(f))\leqslant\epsilon.$$
In particular, if $m=n$, in a small neighborhood the numbers of roots counting multiplicities of $f,g$ are the same.
\end{lem} 
\begin{proof}
Take any point $p\in R(f)$. For ease we think it is $0$. Take a circle $B$ with center $0$ and radius $r<1$ and at most $\epsilon/n$ to separate it from the other (distinct) roots. Let $0<M:=\max_{\partial B}|f(z)|$ since there is no root on this circle. Set
$$h(z):=\delta_mz^{m}+\delta_{m-1}z^{m-1}+...+\delta_{0}.$$
Now, $$|h|\leqslant|\delta_m|r^{m}+|\delta_{m-1}|r^{m-1}+...+|\delta_0|<|\delta_m|+|\delta_{m-1}|+...+|\delta_{0}|.$$
Choose $0<\delta<M/m$. Therefore, for any choice of $\delta_i<\delta$, we have $|h|<|f|$ on $B$. Now, apply \thmref{thm:A.1} for $f$ and $f+h$ which states that they have the same amount of roots in $B$. Apply this procedure for all distinct points in $R(f)$. We will get the desired $\delta$ and $S$.
\end{proof}

\begin{prop}
\label{prop:5.28}
Let $C_1,C_2$ be two curves of degree $d_1,d_2$ intersect transversally at smooth points $p_1,...,p_n$. Then  perturbations on coefficients leave the intersection points staying in a small neighborhood. That is to say for any $\epsilon>0$ there exists $\delta>0$ so that for all curves $C_1',C_2'$ of degree $d_1,d_2$ with $d_{L_{d_1}}^{FS}(C_1,C_1')<\delta,d_{L_{d_2}}^{FS}(C_2,C_2')<\delta$ we have $$\sum_{i}d_{2}^{FS}(C_1'\cap C_2',p_i)<\epsilon.$$
\end{prop}
\begin{proof} Without loss of generality, we can assume $C_1,C_2$ intersect in an affine chart. Then we will study the affine equation. Pick $n$ small open balls $B_n$ of radius at most $\epsilon$ centered at $p_1,...,p_n$ such that they are disjoint from each other. Pick any of these ball and assume its center is $(0,0)$ for simplicity. $C_1\smallsetminus \cup_iB_i$ and $C_2\smallsetminus \cup_iB_i$ are compact. Cover $C_1$ and $C_2$ with with finitely many small cubes to separate them which is possible because the space satisfies separation axiom $T_4$.

We claim the two curves will stay in the union of the small enough cubes ($\delta$-neighborhood). Otherwise, we assume there is a sequence of curves $C_{1,i}$ such that $\lim_{i\apc\infty}C_{1,i}=C_1$. But there are always points $q_i\in C_{0}^{(i)}$ such that $q_i$ lies outside the union of these cubes. Say there is a cluster point of $\{q_i\}$ being $q^*$ lying outside the union of these cubes. The existence is guaranteed by the (sequentially) compactness of the space. There must be one leading term on which the all curves does not vanish (for $i$ big enough). 

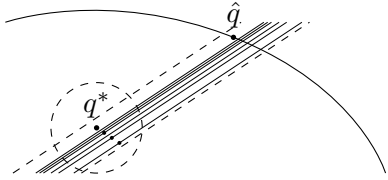
\begin{center}
\begin{tikzpicture}
\draw plot [smooth,tension=1.2] coordinates {(0,2) (3,1.8) (5,0)};
\fill[fill=black] (3,1.8) circle (1pt) node[above]{$\hat{q}$};
\fill[fill=black] (1.2,0.6) circle (1pt)node[above]{$q^*$};
\fill[fill=black] (1.3,0.533) circle (0.75pt);
\fill[fill=black] (1.4,0.467) circle (0.75pt);
\fill[fill=black] (1.5,0.4) circle (0.75pt);
\draw [line width=0.01mm](0.4,0) -- node[above]{} (3.4,2);
\draw [line width=0.01mm](0.45,0) -- node[above]{} (3.45,2);
\draw [line width=0.01mm](0.5,0) -- node[above]{} (3.5,2);
\draw [line width=0.01mm](0.6,0) -- node[above]{} (3.6,2);
\draw [line width=0.02mm](0.7,0) -- node[above]{} (3.7,2);
\draw [line width=0.02mm](0.9,0) -- node[above]{} (3.9,2);
\draw [dashed](1,0) -- node[above]{} (4,2);
\draw [dashed](0.1,0) -- node[above]{} (3.1,2);
\draw [dashed] (1.2,0.6) circle (0.6cm);
\end{tikzpicture}
\end{center}

We draw a thin channel around the point $q^*$. The direction of the channel is chosen in the way that drawing parallel lines passing through these $q_{i}$'s along the channel will intersect all these curves and keep the degree of these curves $\{C_i\}$ and $C$. This is possible because the number fail to keep the degree will have measure $0$ on the complex plane. The univariate complex equations of resulting polynomials cutting $C_{1,i}$ will converge to the univariate complex polynomial defining by cutting the $C_1$ by the passing through $q^*$. But these polynomials give a root does not converge to any other roots of the original polynomial. This is a direct contradiction to \lemref{lem:5.27}.

Now, all intersection points must be retained in these balls $B_k$. We want to show the nodes can not be ``teleported". Assume the point is $(0,0)$. We may then write the two curves as $aX+bY+\text{higher degree}$ and $cX+dY+\text{higher degree}$ by assumption we have $ad-bc\neq0$. This means: under perturbation the curves remain smooth and transversal if intersect at all in the small neighborhood. Again by the channel method described above we can show that the number of nodes will not increase in each ball: assume there is a sequence of pairs of curves $C_{1,i}$ and $C_{2,i}$ with $\lim_{i\apc\infty}C_{1,i}=C_1$ and $\lim_{i\apc\infty}C_{2,i}=C_2$. Such in some ball $B_k$ containing node $p_k$, the number of nodes must be more than one. Assume there are two sequences $\{q_i\}$ and $\{s_i\}$ approaching to the node and each pair $(p_i, s_i)$ are the intersection points of the two curves $C_{1,i}$ and $C_{2,i}$ in the neighborhood and can be assume to be distinct. Choose a thin channel so that the line in the range is transversal to $C_1$ at $p_i$. Then draw lines $Q_i$ in the following way: First, determine a line $l$ which passing through the node $p_k$ whose direction is around tangent direction of $C_2$ such that $l$ and $C_1$ are transversal the polynomial resulted by intersection between $C_1$ and a line keeps the degree of the $C_1$. We will get a line passing through $p_k$. Second, draw a line $l_i'$ through $q_i$ and $s_i$ whose direction should be around the direction of the tangent line of $C_1$, re-orient $l_i'$ in a small range to get a new line $l_i$ such that its direction is around the direction of $l$ within $1/i$, the resulting polynomial preserves the degree of $C_{1,i}$, and the intersection between $l_i$ and $C_{1,i}$ will be near to $q_i$ and $s_i$, i.e., still in the ball $B_i$. This is possible due to \lemref{lem:5.27}. For all $C_{1,i}$ and $l_i$ the line intersects at two points (at least), but the limit is $C_1$ and $l$ which is transversal to a line passing through $p_k$. All the resulting polynomials are of the same degree and having coefficients sufficiently near. Apply \lemref{lem:5.27}, we get the case is not possible. But the total number of intersection points is constant due to \thmref{thm:A.2}. We can conclude.
\end{proof}

\begin{hyp}
Let $C_1,C_2$ be two curves of degree $d_1,d_2$ intersect points $p_1,...,p_n$. The perturbation on coefficients leave the intersection number staying in a small neighborhood.
\end{hyp}

\begin{prop}
\label{prop:5.30}
$D_{d,n}$ is smooth of dimension $L_d-n$.
\end{prop}
\begin{proof}
We consider $\IP^{L_d}$ and $\IP^{L_d}\times(\IP^2)^n$ as Riemannian manifolds naturally. Now, $\tilde\pi:K_{d,n}/S_n\to\C\IP^{L_d}$ is bijective onto the image and smooth.

It is further an immersion. Ascribed by the matrix in the proof of \propref{prop:2.8}, we can see that after projection the tangent space of $K_{d,n}$ is mapped injectively -- $K_{d,n}$ is written locally as the variables in the kernel of the $F$ in \propref{prop:2.8}. Locally, the tangent maps isomorphically onto the tangent space of the locus of curves that passing through $p_1,...,p_n$.

The inverse is continuous. Pick any curve $C\in\C\IP^{L_d}$. Let $\partial_x C$ and $\partial_y C$ be the two partial derivatives, also curves, of the curve. They are smooth and intersect transversally at each node. For any curve $C'$ in a small neighborhood of $C$, its derivatives are also near to the original $\partial_x C$ and $\partial_y C$. The intersection of derivatives lies in a sufficient small neighborhood of the node due to \propref{prop:5.28}. However, there will not be more. Hence, the nodes will be around the original nodes. Of course, there can be less nodes. But in our case, the manifold represents only the curves with less nodes. Therefore, the number of nodes will not decrease. For any $\epsilon>0$, it is easy to see that there exists a neighborhood of $C$ in $D_{d,n}$ such that all curves lie in that neighborhood have nodes near to the nodes of $C$. Take the product metric on $\IP^{L_d}\times(\IP^2)^n$. Then the continuity is proved.\\
This means $\tilde\pi$ is an embedding, namely, the image $D_{d,n}$ is an embedding submanifold of $\C\IP^{L_d}$. As a subset of a (Zariksi) closed set $\overline{D_{d,n}}$ in the projective space $\IP^{L_d}$, each point in $D_{d,n}$ is smooth.
\end{proof}

\begin{prop}
\label{prop:5.31}
Let $f:X\to Y$ be a morphism of two regular quasi-projective schemes. If $\dim_x(X)=\dim_{f(x)}(Y)+\dim_x(X_y)$, then $f$ is flat at $x$.
\end{prop}
\begin{proof}
A quasi-projective scheme is of finite type, hence, Noetherian. Then each stalk is a Noetherian local ring. Take the stalk at $x,f(x)$. Set $R:=\SO_{X,x}$ and $S:=\SO_{Y,f(y)}$ with maximal ideals $\mathfrak{m}$ and $\mathfrak{n}$ respectively. Because the schemes are regular, then $R$ and $S$ are regular with $\dim(R)=\dim_x(X)$ and $\dim(S)=\dim_{f(x)}(Y)$. But we will weaken the assumption on $S$ to be a Cohen-Macaulay ring. And due to the condition, we have
$$\dim(R)=\dim(S)+\dim(R/\mathfrak{m}R).$$
To show is the induced morphism $f_{x}:S\to R$ is flat. We prove it by induction on dimension of $S$.

If $\dim(S)=0$, then $S$ is a field. Then it is trivial.

Assume the statement is true for any such $S$ whose dimension is less than $n$. Now, for the case where $\dim(S)=n+1$, take $v\in\mathfrak{m}\setmi\mathfrak{m}^2$. $A:=S/vS$ and $B:=R/vR$ are still Noetherian local rings with maximal ideals $\mathfrak{a}$ and $\mathfrak{b}$. Let $(s_1,...,s_{\dim{S}})$ be a system of parameter of $S$ and $(r_1,...,r_{\dim{R}})$ be a system of parameters of $R$. Apparently, $\dim(A)=\dim(S)-1$ with a system of parameter $(a_1,...,a_{\dim(A)})$. And the number of elements in the system of parameters $(b_1,...,b_{\dim(B)})$ of $B$ must be at least $\dim{R}-1$, i.e., $\dim(B)\geqslant\dim(R)-1$.  On the other, choose $c_1,...,c_s\in B$ so that $([c_1],...,[c_s])$ form a system of parameter of $B/\mathfrak{a}B$. For a sufficient large power, $\mathfrak{b}^k\subseteq\sum b_iB+\sum c_jB$. This means $\dim(B)\leqslant\dim(A)+\dim(B/\mathfrak{a}B)$. By the isomorphism theorem, $R/\mathfrak{m}R\cong (R/vR)/(\mathfrak{m}R/vR)\cong B/\mathfrak{a}B$. Hence,
$$\dim(B)\leqslant\dim(A)+\dim(B/\mathfrak{a}B)=\dim(A)+\dim(B/\mathfrak{a}B)=$$
$$\dim(S)-1+\dim(R/\mathfrak{m}R)=\dim(R)-1.$$
Put all together $\dim(B)=\dim(R)-1$. Now, we still have $A$ is regular and $B$ is Cohen-Macaulay with the same equality. Therefore, $B$ is flat over $A$. Hence, $\Tor^{S}_1(S/\mathfrak{m}S,R)\cong\Tor^{A}_1(S/\mathfrak{m}S,B)\cong0$. Then $R$ is flat over $S$. 
\end{proof}

\begin{thm}
$\pi:E_{d,n}\to \overline{B_{d,n}}$ is \'etale at each point in $K_{d,n}$.
\end{thm}
\begin{proof}
It follows right away from a proposition in \cite{GRO67} P71 that a flat morphism $f:X\to Y$ satisfies that for all point $x\in X$ whose fiber at each point $f(x)=y\in Y$ is a disjoint union of finite separable extension of the residue field $\kappa(y)$. Due to \propref{prop:5.30} the dimension at each point $D_{d,n}$ is the same as the dimension of each point in $K_{d,n}$ and they are smooth. Apply \propref{prop:5.31}, we get the desired flatness. Hence, it is \'etale.
\end{proof}

\newpage

\appendix

\section{Cited Statements}
\begin{thm}
\label{thm:A.1}
\emph{(Rouch\'e)} Given $D_r(z_0)\subseteq\Omega\subseteq\C$ and $f,g:\Omega\to\C$ holomorphic, if $|g|<|f|$ on $\partial D_r(z_0)$ then the number of solutions counting multiplicities of $f$ and $f+g$ are the same.
\end{thm}

\begin{thm}
\label{thm:A.2}
\emph{(B\'ezout)} Let $C,D$ be two plane curves determined by two \emph{(}homogeneous\emph{)} polynomial equations of degree $c$ and $d$ respectively. Then they intersect at $cd$ points counting with multiplicities or transversally.
\end{thm}

\section{Notation}

\begin{multicols}{2}
$\N:=\{0,1,2,3,...\}$\\
$\Z^+:=\N^+:=\{1,2,3,4,...\}$\\
$\R^+:=\{r\in\R\mid r>0\}$\\
$\R^+_0:=\{r\in\R\mid r\geqslant0\}$\\
$R^{\times}$: group of units of $R$\\
$B_r(z):=\{x\mid d(z,x)<r\}$\\
$D_r(z):=\{x\mid d(z,x)\leqslant r\}$\\
$J_p(\mathbf{f})$: Jacobi matrix of $\mathbf{f}$ at $p$\\
$\imof{\bullet}$: image or preimage of a subset\\
\end{multicols}

\newpage

\bibliographystyle{amsalpha}

\begin{thebibliography}{}
\bibitem[ACG98]{ACG85} E. Arbarello \& M. Cornalba \& P. A. Griffiths \& J. Harris, \emph{Geometry of Algebraic Curves I}, Springer-Verlag Heidelberg, 2011
\bibitem[BES78]{BES78} A. L. Besse, \textit{Manifolds all of whose Geodesics are Closed}, \textit{Springer Verlag}, 1978
\bibitem[GRO67]{GRO67} A. Grothendieck, \textit{\'El\'ements de G\'eom\'etrie Alg\'ebrique IV}, \textit{Publication Math\'emaques de L'I.H.\'E.S.}, 1967
\bibitem[EIS95]{EIS95} D. Eisenbud, \emph{Commutative Algebra with a View toward Algebraic Geometry}, Springer-Verlag New York, 1995
\bibitem[FUL98]{FUL98} W. Fulton, \emph{Intersection Theory}, Springer-Verlag New York, 1998
\bibitem[HAR77]{HAR77} R. Hartshorne, \emph{Algebraic Geometry}, Springer-Verlag New York, 1977
\bibitem[HAR86]{HAR86} J. Harris, \emph{On the Severi Problem}, Inventiones Mathematicae, 1986
\bibitem[HAM98]{HAM98} J. Harris \& I. Morrison, \emph{Moduli of Curves}, Springer-Verlag New York, 1998
\bibitem[SEV21]{SEV21} F. Severi, \emph{Vorlesung \"uber Algebraische Geometrie}, Springer Fachmedien Wiesbaden GmbH, 1921
\bibitem[STA19]{STA19} Stack Project Authors, \emph{The Stack Project}, Cornell University, 2019
\bibitem[STS03]{STS03} E. M. Stein \& R. Shakarchi, \emph{Complex Analysis}, Princeton University Press, 2003
\end{thebibliography}
\providecommand{\bysame}{\leavevmode\hbox to3em{\hrulefill}\thinspace}
\newpage
\providecommand{\href}[2]{#2}

\vspace{25mm}
\hspace{50mm}
\begin{minipage}[]{100mm}
Xiao Yang\\[1mm]
ETHZ\\ [2mm]
{\tt yangx@student.ethz.ch} 
\end{minipage}

\end{document}